\documentclass[10pt]{article}

\usepackage{amsmath}    
\usepackage{graphicx}   
\usepackage{verbatim}   
\usepackage{color}      
\usepackage{hyperref}   

\usepackage{amssymb}
\usepackage{amsthm}

\newtheorem{Def}{Definition}[section]

\newtheorem{Prop}[Def]{Proposition}
\newtheorem{Lem}[Def]{Lemma}
\newtheorem{Thm}[Def]{Theorem}
\newtheorem{Cor}[Def]{Corollary}
\newtheorem{Ass}[Def]{Assumption}
\theoremstyle{definition}
\newtheorem{Rem}[Def]{Remark}

\newcommand{\p}{\mathbb{P}}
\newcommand{\e}{\mathbb{E}}
\newcommand{\real}{\mathbb{R}}
\newcommand{\n}{\mathbb{N}}

\newcommand{\1}{{\bf 1}}

\newcommand{\rd}{\mathrm{d}}
\allowdisplaybreaks

\begin{document}

\title{
	On the Euler--Maruyama scheme for degenerate stochastic differential equations with non-sticky condition
}
\author{
	Dai Taguchi
	\footnote{
		Graduate School of Engineering Science,
		Osaka University,
		1-3, Machikaneyama-cho, Toyonaka,
		Osaka, 560-8531, Japan,
		Email: \texttt{dai.taguchi.dai@gmail.com}
	}
	\quad and \quad
	Akihiro Tanaka
	\footnote{
		Graduate School of Engineering Science,
		Osaka University,
		1-3, Machikaneyama-cho, Toyonaka,
		Osaka, 560-8531, Japan
		/
		Sumitomo Mitsui Banking Corporation,
		1-1-2, Marunouchi, Chiyoda-ku, 
		Tokyo, 100-0005, Japan,
		Email: \texttt{tnkaki2000@gmail.com}
	}
}
\maketitle
\begin{abstract}
The aim of this paper is to study weak and strong convergence of the Euler--Maruyama scheme for a solution of one-dimensional degenerate stochastic differential equation $\mathrm{d} X_t=\sigma(X_t) \mathrm{d} W_t$ with non-sticky condition.
For proving this, we first prove that the Euler--Maruyama scheme also satisfies non-sticky condition.
As an example, we consider stochastic differential equation $\mathrm{d} X_t=|X_t|^{\alpha} \mathrm{d} W_t$, $\alpha \in (0,1/2)$ with non-sticky boundary condition and we give some remarks on CEV models in mathematical finance.
\end{abstract}

\section{Introduction}
Let $X=(X_t)_{t \in [0,T]}$ be a solution of one-dimensional stochastic differential equations (SDEs)
\begin{align}\label{SDE_1}
	\rd X_t
	=\sigma(X_t)\rd W_t,
	~t \in [0,T],~X_0=x_0\in \real,
\end{align}
where $W=(W_t)_{t \in [0,T]}$ is a one-dimensional standard Brownian motion on a probability space $(\Omega, \mathcal{F},\p)$ with a filtration $(\mathcal{F}_t)_{t\geq 0}$ satisfying the usual conditions.
It is well-known that if the coefficient $\sigma$ is Lipschitz continuous then a solution of the equation \eqref{SDE_1} can be constructed by a limit of Picard's successive approximation, and the solution satisfies the pathwise uniqueness.

In the one-dimensional setting, Engelbert and Schmidt \cite{EnSc84} provided an equivalent condition on $\sigma$ for the existence of a weak solution and uniqueness in law for SDE \eqref{SDE_1}, by using time change of a Brownian motion (see also \cite{EnHe80}).
More precisely, they proved that the equation \eqref{SDE_1} has a non-exploding weak solution for every initial condition $X_0=x_0 \in \real$ if and only if $I(\sigma) \subset Z(\sigma)$, and the solution is unique in the sense of probability law if and only if $I(\sigma) = Z(\sigma)$.
Here the sets $I(\sigma)$ and $Z(\sigma)$ are defined by
$
	I(\sigma)
	:=
	\{
		x \in \real
		;
		\int_{-\varepsilon}^{\varepsilon}
			\sigma(x+y)^{-2}
		\rd y
		=
		+\infty
		,~\forall \varepsilon>0
	\}
$
and
$
	Z(\sigma)
	:=
	\{
	x\in \real
	~;~
	\sigma(x)=0
	\}.
$
However, there exists a function $\sigma$ such that $I(\sigma) \varsubsetneq Z(\sigma)$, so in this setting, the uniqueness in law does not hold.
For example, if $\sigma(x):=|x|^{\alpha}$ for $\alpha \in (0,1/2)$ then $I(\sigma)=\emptyset$, $Z(\sigma)=\{0\}$, and if $x_0=0$ then $X_t=0$ and the time change a Brownian motion are solutions of the SDE, and moreover, if $x_0 \neq 0$, there is a solution which spends at zero.
Therefore, as a concept of a solution of SDE, Engelbert and Schmidt \cite{EnSc85} introduced a {\it fundamental solution} of the equation \eqref{SDE_1}, which is a solution of SDE \eqref{SDE_1} with the following {\it non-sticky condition}:
\begin{align}\label{boundary_0}
	\e\left[
		\int_{0}^{T}
			\1_{Z(\sigma)}(X_s)
		\rd s
	\right]
	=0,
\end{align}
that is, in other words, $\sigma^2(X_s(\omega))>0$, $\mathrm{Leb} \otimes \p$-a.e., and proved that there exists a weak solution for a fundamental solution of SDE \eqref{SDE_1} and uniqueness in law holds (see, Theorem 5.4 in \cite{EnSc85}).

On the other hand, the pathwise uniqueness for a solution of SDE is an important concept of a uniqueness for a solution of SDE.
Yamada and Watanabe \cite{YaWa} proved that the pathwise uniqueness implies uniqueness in the sense of probability law, and weak existence and pathwise uniqueness imply the solution is a strong solution.
Moreover, they also showed that under one-dimensional setting, if the diffusion coefficient $\sigma$ is $\alpha$-H\"older continuous with exponent $\alpha \in [1/2,1]$, then the pathwise uniqueness holds (see also \cite{LeGall} and \cite{Nakao} for dis-continuous setting of $\sigma$).
Besides, Girsanov \cite{Gi62} and Barlow \cite{Ba82} provided some examples of $\alpha$-H\"older continuous function $\sigma$ with $\alpha \in (0,1/2)$ such that the pathwise uniqueness fails for SDE \eqref{SDE_1}, and thus the H\"older exponent $\alpha = 1/2$ is sharp.

Under such a background on the pathwise uniqueness, Manabe and Shiga \cite{MaShi73} studied a solution of SDEs with non-sticky boundary condition $\e[\int_{0}^{T}\1_{\{0\}}(X_s)\rd s]=0$ (see also page 221 of \cite{IkWa}).
They proved that if the diffusion coefficient $\sigma$ is bounded, continuous and odd function and continuously differentiable on $\real \setminus \{0\}$ such that (i) $\int_{0}^{\delta}\sigma(y)^{-2} \rd y<\infty$ for some $\delta>0$ and (ii) the limit $\lim_{x \searrow 0}x a'(x) a(x)^{-1}$ exists and is not $1/2$, then two solutions $X^1$ and $X^2$ of SDE \eqref{SDE_1} with non-sticky boundary condition $\e[\int_{0}^{T}\1_{\{0\}}(X_s)\rd s]=0$ and the same initial value and driven by the same Brownian motion, satisfy $\p(|X_t^1|=|X_t^2|,~\forall t\geq 0)=1$.
However, sign of solutions does not know from the information of driving Brownian motion.
Moreover, additionally if $\sigma(0)=0$ then the pathwise uniqueness holds for the following reflected SDE with non-sticky boundary condition
\begin{align}\label{RSDE_1}
	X_t
	=
	x_0
	+
	\int_{0}^{t}
		\sigma(X_s)
	\rd W_s
	+L_t^0(X)
	\geq 0,
	~
	\e\left[
		\int_{0}^{t}
			\1_{\{0\}}(X_s)
		\rd s
	\right]
	=
	0,
	~t \in [0,T],~x_0 \geq 0,
\end{align}
where $L^0(X)$ is a local time of $X$ at the origin.
Recently, these results were extended by Bass and Chen \cite{BaCh} and Bass, Burdzy and Chen \cite{BaBuCh}.
It was shown in \cite{BaBuCh} (resp. \cite{BaCh}) that if $\sigma(x)=|x|^{\alpha}$ with $\alpha \in (0,1/2)$ then a strong solution of SDE \eqref{SDE_1} with non-sticky boundary condition $\e[\int_{0}^{T}\1_{\{0\}}(X_s)\rd s]=0$ (resp. reflected SDE \eqref{RSDE_1}) exists and the pathwise uniqueness holds by using excursion theory and pseudo-strong Markov property (resp. approximation argument).
Note that in the case of $\alpha=0$, that is, $\sigma(x)=\1(x \neq 0)$, Pascu and Pascu \cite{PaPa11} studies sticky and non-sticky solutions.

Under the viewpoint of numerical analysis, we often use the Euler--Maruyama scheme $X^{(n)}=(X^{(n)}_t)_{t \in [0,T]}$ which is a discrete approximation for a solution of SDE \eqref{SDE_1} defined by
$
	\rd 
	X_t^{(n)}
	=
	\sigma(X_{\eta_n(t)}^{(n)})
	\rd W_t,~
	X_0^{(n)}=x_0,~t \in [0,T],
$
where $\eta_n(s):=t_{k}^{(n)}=kT/n$, if $s \in [t_{k}^{(n)},t_{k+1}^{(n)})$.
It is well-known that if the coefficient $\sigma$ is Lipschitz continuous, then $X^{(n)}$ has strong ($L^p$-$\sup$) rate of convergence 1/2, that is, $\e[\sup_{0 \leq t \leq T}|X_t-X_t^{(n)}|^{p}]^{1/p} \leq C n^{-1/2}$ for any $p \geq 1$, (see, e.g. \cite{KP}).
On the other hand, the Euler--Maruyama scheme can be applied to many directions not only numerical analysis.
Indeed, Maruyama \cite{Mar54} used the scheme for proving Girsanov's theorem for one-dimensional SDE $\rd X_t=b(X_t)\rd t + \rd W_t$.
Moreover, Skorokhod constructed a (weak) solution of SDE with continuous and linear growth coefficients as a limit of the Euler--Maruyama scheme (see, chapter 3, section 3 in \cite{Sk65}).
Skorokhod's arguments can be also applied to a construction of a solution, which is based on the approximation argument of the coefficients, (see, e.g. chapter 3 in \cite{TaHa64}).
On the other hand, Yamada \cite{Ya78} proved that if the diffusion coefficient $\sigma$ is $\alpha$-H\"older continuous with $\alpha \in [1/2,1]$, then the Euler--Maruyama scheme $X^{(n)}$ converges to the unique strong solution of SDE in $L^2$-$\sup$ sense.
Recently, in the same setting, the rate of convergence was provided (see, \cite{GyRa} and \cite{Yan}), by using Yamada and Watanabe approximation arguments or It\^o--Tanaka formula.
The result of Yamada \cite{Ya78} also extended by Kaneko and Nakao \cite{KaNa}.
They showed that by using the similar arguments of Skorokhod \cite{Sk65}, if the pathwise uniqueness holds for SDE with continuous and linear growth coefficients, then the Euler--Maruyama scheme and a solution of SDE with smooth approximation of the coefficients converge to the solution of corresponding SDE in $L^2$-$\sup$ sense.
For results on weak convergence, when the uniqueness in law holds for SDE with dis-continuous coefficients, then Yan \cite{Yan} provided some equivalent conditions for the weak convergence of the Euler--Maruyama scheme, by using a limit theorem of stochastic integrals.
Moreover, recently, Ankirchner, Kruse and Urusov \cite{AnKrUr18} proved that the weak convergence of the Euler--Maruyama scheme with continuous diffusion coefficient $\sigma$ such that $I(\sigma) = Z(\sigma)=\emptyset$.

Inspired by the above previous works, in this paper, we study weak and strong convergence of the Euler--Maruyama scheme for a solution of SDE \eqref{SDE_1} with non-sticky condition \eqref{boundary_0}.
We first prove that the Euler--Maruyama scheme $X^{(n)}$ defined below (see, \eqref{EM_0}) also satisfies the non-sticky condition \eqref{boundary_0}.
As an application of this fact, we prove that the Euler--Maruyama scheme converges weakly to a unique non-sticky weak solution of SDE, and if the pathwise uniqueness holds then it converges to a unique non-sticky strong solution of SDE in $L^p$-$\sup$ sense for any $p \geq 1$.
The idea of proof is also based on arguments of Yan \cite{Yan} and Skorokhod \cite{Sk65}, and prove that by using occupation time formula if the limit of sub-sequence of the Euler--Maruyama scheme exists, then the limit satisfies the non-sticky condition (see, Lemma \ref{Lem_2} below).
As an example, the unique strong solution of SDE $\rd X_t=|X_t|^{\alpha}\rd W_t$ with non-sticky boundary condition $\e[\int_{0}^{T}\1_{\{0\}}(X_s)\rd s]=0$ for $\alpha \in (0,1/2)$ can be approximated by the Euler--Maruyama scheme. 


This paper is structured as follows.
In section \ref{sec_2}, we prove the weak (resp. strong) convergence for the the Euler--Maruyama scheme to a solution of SDE with non-sticky condition by using the uniqueness in law (resp. pathwise uniqueness).
In subsection \ref{sub_sec_2_1}, we provide the definition of the Euler--Maruyama scheme and prove that it satisfies the non-sticky condition.
In subsection \ref{sub_sec_2_2}, we state the main theorems of this present paper.
We prove some auxiliary estimates in subsection \ref{sub_sec_2_3} and provide the proof of main theorems in subsection \ref{sub_sec_2_4}.

\subsection*{Notations}
We give some basic notations and definitions used throughout this paper.
For a Lipschitz continuous function $f:\real \to \real$, we define $\|f\|_{\mathrm{Lip}}:=\sup_{x \neq y}\frac{|f(x)-f(y)|}{|x-y|}$.
For a given $T>0$, we denote by $C[0,T]$ the space of continuous functions $w:[0,T] \to \real$ with metric $\rho$ defined by $\rho(w,w')=\sup_{0\leq t \leq T} |w_t-w_t'|$, and by $C_b(C[0,T]^{k};\real)$, $k \in \n$, a continuous function $f:C[0,T] \to \real$ such that $\sup_{w \in C[0,T]^{k}}|f(w)|$ is finite.
We denote the sign function by $\mbox{sgn}(x):=-\1_{(-\infty,0]}(x)+\1_{(0,\infty)}(x)$ for $x \in \real$.
For a measurable function $\sigma:\real \to \real$, we define
$
	I(\sigma)
	:=
	\{
		x \in \real
		;
		\int_{-\varepsilon}^{\varepsilon}
		\sigma(x+y)^{-2}
			\rd y
		=
		+\infty
		,~\forall \varepsilon>0
	\}
$
and
$
	Z(\sigma)
	:=
	\{
	x\in \real
	~;~
	\sigma(x)=0
	\}
$, and we denote by $D(\sigma)$ the set of all dis-continuous points of $\sigma$.
For a continuous semi-martingale $Y=(Y_t)_{t \geq 0}$, we denote $L^x(Y)=(L_t^x(Y))_{t \geq 0}$ the symmetric local time of $Y$ at the level $x \in \real$. 
We may write a solution of SDE \eqref{SDE_1} by expressing $(X,W)$.

\section{Weak and strong convergence for the Euler--Maruyama scheme}
\label{sec_2}

Throughout this paper, we suppose the following assumptions for the diffusion coefficient $\sigma$.
\begin{Ass}
	$\sigma:\real \to \real$ is a measurable function and $Z(\sigma)$ is not the empty set and is a countable set, that is, $\sigma$ is degenerate.
\end{Ass}

\subsection{Euler--Maruyama scheme}\label{sub_sec_2_1}

We define the Euler-Maruyama scheme $X^{(n)}=(X_t^{(n)})_{t \in [0,T]}$ for SDE \eqref{SDE_1} by
\begin{align}\label{EM_0}
	X_t^{(n)}
	=
	x_n
	+\int_{0}^{t}
		\sigma(X_{\eta_n(s)}^{(n)})
	\rd W_s,
\end{align}
where the sequence $\{x_n\}_{n \in \n} \subset \real \setminus Z(\sigma)$ satisfies $\lim_{n \to \infty}x_n = x_0 \in \real$ and $\eta_n(s):=t_{k}^{(n)}=kT/n$, if $s \in [t_{k}^{(n)},t_{k+1}^{(n)})$.
Note that since $Z(\sigma)$ is a countable set, there exists such a sequence $\{x_n\}_{n \in \n}$.
From here, we fix the sequence $\{x_n\}_{n\in\n}$.

\begin{Rem}
	Usually the initial value of the Euler--Maruyama scheme $X^{(n)}_{0}$ is defined by $x_0$.
	However, if $Z(\sigma) \neq \emptyset$ and $X_0^{(n)}=x_0 \in Z(\sigma)$, then $X_t^{(n)}=x_0$ for all $t \in [0,T]$.
	Therefore, in order to approximate a solution of SDE \eqref{SDE_1} with non-sticky condition \eqref{boundary_0}, we need to take an approximate sequence $\{x_n\}_{n \in \n}$ from $\real \setminus Z(\sigma)$.
\end{Rem}

Now we prove that the Euler--Maruyama scheme \eqref{EM_0} satisfies the non-sticky condition.

\begin{Lem}\label{Lem_EM_0}
	For any $n \in \n$, $X^{(n)}$ satisfies the non-sticky condition
	\begin{align*}
		\e\left[
			\int_{0}^{T}
				\1_{Z(\sigma)}(X_s^{(n)})
			\rd s
		\right]
		=0.
	\end{align*}
\end{Lem}
\begin{proof}
	We first prove by induction that for each $k=0,\ldots,n-1$, it holds that
	\begin{align*}
		\p(X_s^{(n)} \notin Z(\sigma))=1,
		~\text{for any}~
		s \in (t_{k}^{(n)},t_{k+1}^{(n)}].
	\end{align*}
	
	Since $X_s^{(n)}=x_n+\sigma(x_n)W_s$ for any $s \in (0,t_{1}^{(n)}]$ and $\sigma(x_n) \neq 0$, we have
	$\p(X_s^{(n)} \in Z(\sigma))=0$, that is, $\p(X_s^{(n)} \notin Z(\sigma))=1$.
	Thus the statement holds for $k=0$.
	
	Now we assume that the statement holds for $\ell=1,\ldots,k-1$.
	Then since
	$
		X_s^{(n)}
		=
		X_{t_k^{(n)}}^{(n)}
		+
		\sigma(X_{t_k^{(n)}}^{(n)})
		(W_s-W_{t_{k}^{(n)}})
	$
	for any $s \in (t_{k}^{(n)},t_{k+1}^{(n)}]$, by the assumption $\p(X_{t_{k}^{(n)}}^{(n)} \notin Z(\sigma))=1$, we have
	\begin{align*}
		\p(X_s^{(n)} \in Z(\sigma))
		&
		=
		\e\left[
			\p\left(
				X_{t_k^{(n)}}^{(n)}
				+
				\sigma(X_{t_k^{(n)}}^{(n)})
				(W_s-W_{t_{k}^{(n)}})
				\in Z(\sigma)
				~\Big|~
				X_{t_k^{(n)}}^{(n)}
			\right)
			\1(X_{t_k^{(n)}}^{(n)} \notin Z(\sigma))
		\right].
	\end{align*}
	Note that random variables $X_{t_k^{(n)}}^{(n)}$ and $W_s-W_{t_{k}^{(n)}}$ are independent, thus we have
	\begin{align*}
	\p(X_s^{(n)} \in Z(\sigma))
	&=
	\e\left[
		\p\left(
			x
			+\sigma(x)
			(W_s-W_{t_{k}^{(n)}})
			\in Z(\sigma)
		\right)
		\Big|_{x=X_{t_k^{(n)}}^{(n)}}
		\1(X_{t_k^{(n)}}^{(n)} \notin Z(\sigma))
	\right]
	=0.
	\end{align*}
	This concludes the case for $k$.
	Hence we have for each $k=0,\ldots,n-1$, it holds that $\p(X_s^{(n)} \notin Z(\sigma))=1$ for any $s \in (t_{k}^{(n)},t_{k+1}^{(n)}]$.
	
	Using this fact, we have
	\begin{align*}
		\e\left[
			\int_{0}^{T}
				\1_{Z(\sigma)}(X_s^{(n)})
			\rd s
		\right]
		=\sum_{k=0}^{n-1}
			\int_{t_{k}^{(n)}}^{t_{k+1}^{(n)}}
				\p(X_s^{(n)} \in Z(\sigma) )
			\rd s
		=0,
	\end{align*}
	which concludes the statement.
\end{proof}

\subsection{Main results}\label{sub_sec_2_2}

In this subsection, we provide a weak and strong convergence for the Euler--Maruyama scheme.

We need the following assumptions on the diffusion coefficient $\sigma$.

\begin{Ass}\label{Ass_1}
	
	\begin{itemize}
		\item[(i)]
		For any $z \in Z(\sigma)$,
		\begin{align*}
			\lim_{\varepsilon \searrow 0}
			\int_{-\varepsilon}^{\varepsilon}
				\frac{1}{\sigma(z+y)^2}
			\rd y
			=0.
		\end{align*}
		\item[(ii)]
		The diffusion coefficient $\sigma$ is of linear growth, (i.e., there exists $K>0$ such that for any $x \in \real$, $|\sigma(x)| \leq K(1+|x|)$), continuous almost everywhere with respect to Lebesgue measure  and $\sigma_1(y)^2>0$ for any $y \in D(\sigma)$, where $\sigma_1(y)^2:=\liminf_{x \to y} \sigma(x)^2$ for $y \in \real$.

	\end{itemize}
\end{Ass}

\begin{Rem}\label{Rem_0}
	\begin{itemize}
		\item[(i)]
		Assumption \ref{Ass_1} (i) implies that for any $z \in Z(\sigma)$, there exists $\varepsilon(z)>0$ such that
		$
		\int_{-\varepsilon(z)}^{\varepsilon(z)} \frac{1}{\sigma(z+y)^2}
		\rd y <\infty
		$,
		thus
		$I(\sigma)=\emptyset \neq Z(\sigma)$.
		Therefore from the result of Engelbert and Schmidt (see, e.g. Theorem 5.5.7 in \cite{KS}), the uniqueness in law does not hold for SDE \eqref{SDE_1}.
		However, it follows from Theorem 5.4 in \cite{EnSc85} that a solution of SDE \eqref{SDE_1} with non-sticky condition \eqref{boundary_0} exists and uniqueness in law holds by using time change of a Brownian motion.
		
		\item[(ii)]
		It follows from Assumption \ref{Ass_1} (i) that if the Euler--Maruyama scheme converges to some stochastic process, almost surely, then the limit satisfies the non-sticky condition, (see, Lemma \ref{Lem_2} (iv)).
	\end{itemize}
	
\end{Rem}

We obtain the following result on the weak convergence of the Euler--Maruyama scheme.

\begin{Thm}\label{main_0}
	Suppose that Assumption \ref{Ass_1} holds.
	Let $X=(X_t)_{0\leq t \leq T}$ be a solution of SDE \eqref{SDE_1} with non-sticky condition \eqref{boundary_0} and $\{X^{(n)}\}_{n \in \n}$ be the Euler--Maruyama scheme for $X$ defined by \eqref{EM_0}.
	Then for any $f \in C_b(C[0,T];\real)$,
	\begin{align*}
		\lim_{n \to \infty}
		\e[f(X^{(n)})]
		=
		\e[f(X)].
	\end{align*}
\end{Thm}

If $\sigma$ is continuous and the pathwise uniqueness holds for $X$, then we have the strong convergence for the Euler--Maruyama scheme.

\begin{Thm}\label{main_1}
	Suppose that Assumption \ref{Ass_1} holds and $\sigma$ is continuous.
	Let $X=(X_t)_{t \in [0,T]}$ be a solution of SDE \eqref{SDE_1} with non-sticky condition \eqref{boundary_0} and $\{X^{(n)}\}_{n \in \n}$ be the Euler--Maruyama scheme for $X$ defined by \eqref{EM_0}.
	
	\begin{itemize}
		\item[(i)]
		If the pathwise uniqueness holds for $X$, then for any $p \in [1,\infty)$,
		\begin{align*}
			\lim_{n \to \infty}
				\e\left[
					\sup_{0 \leq t \leq T}
						\left|
							X_t
							-
							X_t^{(n)}
						\right|^p
				\right]
			=
			0.
		\end{align*}
		\item[(ii)]
		Suppose that $\p(|X_t|=|X_t'|,~\forall t\geq 0)=1$, for any the other solution $X'$ of SDE \eqref{SDE_1} driven by the same Brownian motion, with non-sticky condition \eqref{boundary_0}.
		Then for any $p \in [1,\infty)$,
		\begin{align*}
			\lim_{n \to \infty}
				\e\left[
					\sup_{0 \leq t \leq T}
						\left|
							|X_t|
							-
							|X_t^{(n)}|
						\right|^p
				\right]
			=
			0.
		\end{align*}
	\end{itemize}
\end{Thm}

\subsection{Examples and applications}

As examples of Theorem \ref{main_1}, we have two corollaries.

The first example is an application of a result in \cite{BaBuCh}.

\begin{Cor}\label{Cor_1}
	Let $\sigma(x)=|x|^{\alpha}$, $\alpha \in (0,1/2)$ and $X=(X_t)_{0\leq t \leq T}$ be a solution of SDE \eqref{SDE_1} with non-sticky boundary condition $\e[\int_{0}^{T}\1_{\{0\}}(X_s)\rd s]=0$, and $\{X^{(n)}\}_{n \in \n}$ be the Euler--Maruyama scheme for $X$ defined by \eqref{EM_0}.
	Then for any $p \in [1,\infty)$,
	\begin{align*}
		\lim_{n \to \infty}
		\e\left[
			\sup_{0 \leq t \leq T}
				\left|
					X_t
					-
					X_t^{(n)}
				\right|^p
		\right]
		=
		0.
	\end{align*}
\end{Cor}
\begin{proof}
	From Theorem 1.2 in \cite{BaBuCh}, the pathwise uniqueness holds for SDE $\rd X_t=|X_t|^{\alpha} \rd W_t$, $X_0=x_0 \in \real$ with non-sticky boundary condition.
	On the other hand, since $Z(\sigma)=\{0\}$ and $\alpha \in (0,1/2)$, it holds that
	\begin{align*}
		\lim_{\varepsilon \searrow 0}
		\int_{-\varepsilon}^{\varepsilon}
		\frac{1}{|y|^{2\alpha}}
		\rd y
		=
		\lim_{\varepsilon \searrow 0}
		\frac{2\varepsilon^{1-2\alpha}}{1-2\alpha}
		=0.
	\end{align*}
	Hence $\sigma(x)=|x|^{\alpha}$ satisfies Assumption \ref{Ass_1}.
	From Theorem \ref{main_1}, we conclude the statement.
\end{proof}

We give a financial application of SDE considered in Corollary \ref{Cor_1}.
In mathematical finance, constant elasticity of variance (CEV) models introduced by Cox \cite{Cox96}
\begin{align*}
	\rd X_t=(X_t)^{\alpha} \rd W_t,~
	X_0=x_0>0,~\alpha \in (0,1],
\end{align*}
have been studied by many authors (see, e.g. \cite{AnKoSp15}, \cite{DeSh02}, \cite{GuHoJa18}, \cite{Is09} and \cite{JeYoCh09}).
If $\alpha \in [1/2,1]$, then as mentioned in the introduction, pathwise uniqueness holds (see, Theorem 1 in \cite{YaWa} or Proposition 5.2.13 in \cite{KS}).
Moreover, the boundary point zero is absorbing, that is, the process remains at zero after it reaches zero (see, e.g. Proposition 6.1.3.1 in \cite{JeYoCh09}).

On the other hand, if $\alpha \in (0,1/2)$, then the pathwise uniqueness does not hold for CEV models, and the boundary zero is regular, that is, the solution can get in to the boundary zero and can get out from the boundary zero, (see, e.g. \cite{BoSa} or Example 5.4 in \cite{Du96}).
Therefore one may consider CEV models with absorbing boundary (see, \cite{DeSh02}), or reflecting boundary by setting $X_t:=(1-\alpha)^{\frac{1}{1-\alpha}}(\rho_t)^{\frac{1}{2(1-\alpha)}}$, where $\rho=(\rho_t)_{t \geq 0}$ be a $(1-2\alpha)/(1-\alpha)$-dimensional squared Bessel process (see, the explicit form of the density function given in \cite{JeYoCh09}, page 367, case $\beta(=\alpha-1)<0$).

Recently, there are some studies on CEV models $\rd X_t=|X_t|^{\alpha} \rd W_t$ with ``free boundary condition" (see, e.g. subsection 2.2 in \cite{AnKoSp15}) to extend them as $\real$-valued processes.
However, as mentioned in the introduction, the uniqueness in law and pathwise uniqueness do not hold (in particular, there is no density function) without some boundary conditions.
Therefore, if one would like to extend CEV models as $\real$-valued processes, then as one approach, the non-sticky boundary is useful.

Finally, we give a relation between CEV model with non-sticky boundary condition and squared Bessel process.
Let $X=(X_t)_{t \in [0,T]}$ be a solution of SDE $\rd X_t=|X_t|^{\alpha} \rd W_t$, $X_0=x_0 \in \real$, for $\alpha \in (0,1/2)$.
We first do not assume any boundary condition for $X$.
Let $g(x):=|x|^{2(1-\alpha)}/(1-\alpha)^2$.
Then it is easy to see that
\begin{align*}
	g'(x)
	=\frac{2\mathrm{sgn}(x)}{1-\alpha}|x|^{1-2\alpha},~
	\forall x \in \real
	\quad\text{and}\quad
	g''(x)
	=\frac{2(1-2\alpha)}{1-\alpha}
	\frac{1}{|x|^{2\alpha}},~
	\forall x \in \real \setminus \{0\}.
\end{align*}
So we cannot apply It\^o's formula for $g$, but since $g$ is convex and $X$ is a continuous martingale, we can apply It\^o--Tanaka formula (see, e.g. Theorem 1.5 in chapter VI of \cite{ReYo99}) to obtain
\begin{align*}
	Y_t
	:=g(X_t)
	=
	g(x_0)
	+
	\int_{0}^{t}
		g'(X_s)
	\rd X_s
	+
	\frac{1}{2}
	\int_{\real}
		L_t^x(X)
	g''(\rd x),
\end{align*}
where $g''(\rd x)$ is the second derivative measure of $g$, and is given by
\begin{align*}
	g''(\rd x)
	=
	\frac{2(1-2\alpha)}{1-\alpha}
	\frac{\1_{\real \setminus \{0\}}(x)}{|x|^{2\alpha}}
	\rd x.
\end{align*}
Therefore, the occupation time formula (see, e.g. Corollary 1.6 in chapter VI of \cite{ReYo99}), we have
\begin{align*}
	Y_t
	&
	=
	g(x_0)
	+
	2
	\int_{0}^{t}
		\sqrt{Y_s} \mathrm{sgn}(X_s)
	\rd W_s
	+
	\frac{1-2\alpha}{1-\alpha}
	\int_{0}^{t}
		\frac
			{\1_{\real \setminus \{0\}}(X_s)}
			{|X_s|^{2\alpha}}
	\rd \langle X \rangle_s\\
	&
	=
	g(x_0)
	+
	2
	\int_{0}^{t}
		\sqrt{Y_s} \mathrm{sgn}(X_s)
	\rd W_s
	+
	\frac{1-2\alpha}{1-\alpha}
	\int_{0}^{t}
		\1_{\real \setminus \{0\}}(X_s)
	\rd s.
\end{align*}
We now assume non-sticky boundary condition for $X$, then $\1_{\real \setminus \{0\}}(X_s)=1$ for all $s \in [0,T]$, almost surely and thus
\begin{align*}
	Y_t
	=
	g(x_0)
	+
	2
	\int_{0}^{t}
		\sqrt{Y_s}
	\rd \widetilde{W}_s
	+
	\frac{1-2\alpha}{1-\alpha}t,
\end{align*}
where $\widetilde{W}=(\widetilde{W}_t)_{t \in [0,T]}$ is a Brownian motion defined by
$\rd \widetilde{W}_t:= \mathrm{sgn}(X_t)\rd W_t$.
Therefore, the law of $Y$ is a $(1-2\alpha)/(1-\alpha)$-dimensional squared Bessel process.

\begin{Rem}
	Note that one may use It\^o's formula for $Y_t=|X_t|^{2(1-\alpha)}/(1-\alpha)^2$ for ``some" $t \geq 0$, in order to prove $Y$ satisfies the equation $\rd Y_t=2\sqrt{Y_t} \rd W_t +\frac{1-2\alpha}{1-\alpha} \rd t$, (see, e.g. \cite{Is09} and \cite{GuHoJa18}).
	The above computation shows that this is true for $t <\inf\{s>0;X_s=0\}$.
\end{Rem}

The second example is an application of a result in \cite{MaShi73}.

\begin{Cor}
	Let $Z(\sigma)=\{0\}$, $X=(X_t)_{t \in [0,T]}$ be a solution of SDE \eqref{SDE_1} with non-sticky boundary condition $\e[\int_{0}^{T}\1_{\{0\}}(X_s)\rd s]=0$, and $\{X^{(n)}\}_{n \in \n}$ be the Euler--Maruyama scheme for $X$ defined by \eqref{SDE_1}.
	Suppose that $\sigma:\real \to \real$ satisfies Assumption \ref{Ass_1}, and is a bounded, continuous and odd function and continuously differentiable on $\real \setminus \{0\}$ such that the limit $\lim_{x \searrow 0}x a'(x) a(x)^{-1}$ exists and is not $1/2$.
	Then for any $p \in [1,\infty)$,
	\begin{align*}
		\lim_{n \to \infty}
		\e
		\left[
			\sup_{0 \leq t \leq T}
				\left|
					|X_t|
					-
					|X_t^{(n)}|
				\right|^p
			\right]
		=
		0.
	\end{align*}
\end{Cor}
\begin{proof}
	From Assumption \ref{Ass_1} (ii), there exists $\delta>0$ such that $\int_{0}^{\delta}\sigma(y)^{-2} \rd y<\infty$.
	Hence it follows from Theorem 1 in \cite{MaShi73} that the assumptions on Theorem \ref{main_1} hold.
	Thus we conclude the proof.
\end{proof}

\subsection{Auxiliary estimates}\label{sub_sec_2_3}

In this subsection, we introduce some useful estimates for proving Theorem \ref{main_0} and Theorem \ref{main_1}.

We first prove the following standard inequalities on a solution of SDE \eqref{SDE_1}, the Euler--Maruyama scheme defined by \eqref{EM_0} and their local times.

\begin{Lem}\label{Lem_0}
	Let $X$ be a solution of SDE \eqref{SDE_1} and $\{X^{(n)}\}_{n \in \n}$ be the Euler--Maruyama scheme for $X$ defined by \eqref{EM_0}.
	Suppose that $\sigma$ is of linear growth.
	Then for any $p \geq 1$, there exists a positive constant $C_p>0$ such that
	\begin{align}
		\e\left[
			\sup_{0 \leq t \leq T}
			|X_t|^p
		\right]
		+
		\sup_{n \in \n}
		\e\left[
			\sup_{0 \leq t \leq T}
			|X_t^{(n)}|^p
		\right]
		&
		\leq
		C_p,
		\label{Lem_0_1}\\
		\e[|X_t-X_s|^{p}]^{1/p}
		+
		\sup_{n \in \n}
		\e[|X_t^{(n)}-X_s^{(n)}|^{p}]^{1/p}
		&
		\leq
		C_p
		|t-s|^{1/2},
		~\text{for any}~t,s \in [0,T].
		\label{Lem_0_3}
	\end{align}
	Moreover, there exists $C_0>0$ such that
	\begin{align}
		\sup_{y \in \real}
		\e[L_T^{y}(X)]
		+
		\sup_{y \in \real,~n \in \n}
		\e[L_T^{y}(X^{(n)})]
		&\leq
		C_0
		\label{Lem_0_2}.
	\end{align}
\end{Lem}
\begin{proof}
	Since $\{x_n\}_{n \in \n}$ is bounded and $\sigma$ is of linear growth, the estimates \eqref{Lem_0_1} and \eqref{Lem_0_3} can be shown by applying Gronwall's inequality and Burkholder-Davis-Gundy's inequality, thus it will be omitted.
	
	We prove \eqref{Lem_0_2}.
	By It\^o--Tanaka formula, we have for any $y \in \real$,
	\begin{align*}
		L_T^y(X)
		&=
		|X_T-y|
		-|x_0-y|
		-\int_{0}^{T}
			\mbox{sgn}(X_s-y)
		\rd X_s\\
		&
		\leq
		|x_0|
		+|X_T|
		+
		\left|
			\int_{0}^{T}
			\mbox{sgn}(X_s-y)
			\sigma(X_s)
			\rd W_s
		\right|
	\end{align*}
	and by the same way
	\begin{align*}
		L_T^y(X^{(n)})
		\leq
		|x_n|
		+|X_T^{(n)}|
		+
		\left|
		\int_{0}^{T}
			\mbox{sgn}(X_s^{(n)}-y)
			\sigma(X_{\eta_n(s)}^{(n)})
		\rd W_s
		\right|.
	\end{align*}
	Hence by using Burkholder-Davis-Gundy's inequality and \eqref{Lem_0_1} with $p=1,2$, we conclude \eqref{Lem_0_2}.
\end{proof}

The following lemma is a key estimate for the non-sticky condition.

\begin{Lem}\label{Lem_1}
	Suppose that Assumption \ref{Ass_1} hold.
	Let $X$ be a solution of SDE \eqref{SDE_1} with non-sticky condition \eqref{boundary_0} and $\{X^{(n)}\}_{n \in \n}$ be the Euler--Maruyama scheme for $X$ defined by \eqref{EM_0}.
	Let $z \in Z(\sigma)$ and $f_z:\real \to [0,1]$ be a Lipschitz continuous function with $\mathrm{supp}f_z \subset [z-\varepsilon, z+\varepsilon]$ for some $\varepsilon>0$.
	Then there exists $C>0$ which does not depend on $n$, $z$ and $\varepsilon$ such that
	\begin{align}\label{Lem_1_1}
		\e\left[
			\int_{0}^{T}
				f_z(X_s)
			\rd s
		\right]
		\leq
		C
		\int_{-\varepsilon}^{\varepsilon}
			\frac{1}{\sigma(z+y)^2}
		\rd y
	\end{align}
	and
	\begin{align}\label{Lem_1_2}
		\e\left[
			\int_{0}^{T}
				f_z(X_s^{(n)})
			\rd s
		\right]
		\leq
		C
		\left\{
			\int_{-\varepsilon}^{\varepsilon}
				\frac{1}{\sigma(z+y)^2}
			\rd y
			+
			\frac{\|f_z\|_{\mathrm{Lip}}}{n^{1/2}}
		\right\}.
	\end{align}
\end{Lem}
\begin{proof}
	We first prove \eqref{Lem_1_1}.
	Since $X$ satisfies the non-sticky condition \eqref{boundary_0}, we have $\sigma(X_s(\omega))^2>0$, $\mathrm{Leb} \otimes \p$-a.e.
	Thus by using Fatou's lemma and the occupation time formula (see, e.g. Corollary 1.6 in chapter VI of \cite{ReYo99}) and Lemma \ref{Lem_0}, we have
	\begin{align*}
		\e\left[
			\int_{0}^{T}
				f_z(X_s)
			\rd s
		\right]
		&
		=
		\e\left[
			\int_{0}^{T}
				\frac{f_z(X_s)\1_{\{\sigma(X_s)^2>0\}}}{\sigma(X_s)^2}
			\rd \langle X \rangle_s
		\right]\\
		&
		\leq
		\liminf_{N \to \infty}
		\e\left[
			\int_{\real}
				\frac{f_z(y)\1_{\{\sigma(y)^2>1/N\}}}{\sigma(y)^2}
				L^y_T(X)
			\rd y
		\right]\\
		&\leq
		\sup_{y \in \real}
		\e\left[
			L^y_T(X)
		\right]
		\int_{z-\varepsilon}^{z+\varepsilon}
			\frac{1}{\sigma(y)^2}
		\rd y\\
		&\leq
		C_0
		\int_{-\varepsilon}^{\varepsilon}
		\frac{1}{\sigma(z+y)^2}
		\rd y,
	\end{align*}
	which implies \eqref{Lem_1_1}.

	Now we prove \eqref{Lem_1_2}.
	Since, from Lemma \ref{Lem_EM_0}, we have $\sigma(X_s^{(n)}(\omega))^2>0$, $\mathrm{Leb} \otimes \p$-a.e.
	Hence by using Lipschitz continuity of $f_z$, Fatou's lemma, the occupation time formula and Lemma \ref{Lem_0}, we have
	\begin{align*}
		\e\left[
			\int_{0}^{T}
				f_z(X_s^{(n)})
			\rd s
		\right]
		&\leq
		\e\left[
			\int_{0}^{T}
				f_z(X_{\eta_n(s)}^{(n)})
			\rd s
		\right]
		+
		\int_{0}^{T}
			\e\left[
					\left|
					f_z(X_s^{(n)})
					-
					f_z(X_{\eta_n(s)}^{(n)})
					\right|
			\right]
		\rd s\\
		&\leq
		\e\left[
			\int_{0}^{T}
				\frac
					{f_z(X_{\eta_n(s)}^{(n)}) \1_{\{\sigma(X_{\eta_n(s)}^{(n)})^2>0\}}}
					{\sigma(X_{\eta_n(s)}^{(n)})^2}
			\rd \langle X^{(n)} \rangle_s
		\right]
		+\|f_z\|_{\mathrm{Lip}}
		\int_{0}^{T}
			\e[|X_{s}^{(n)}-X_{\eta_n(s)}^{(n)}|]
		\rd s\\
		&\leq
		\liminf_{N \to \infty}
		\e\left[
			\int_{\real}
				\frac{f_z(y) \1_{\{\sigma(y)^2>1/N \}}}{\sigma(y)^2}
				L^y_T(X^{(n)})
			\rd y
		\right]
		+
		\frac{C_1T^{3/2}\|f_z\|_{\mathrm{Lip}}}{n^{1/2}}\\
		&\leq
		\sup_{n,\in \n,~y \in \real}
		\e\left[
			L^y_T(X^{(n)})
		\right]
		\int_{z-\varepsilon}^{z+\varepsilon}
			\frac{1}{\sigma(y)^{2}}
		\rd y
		+
		\frac{C_1T^{3/2}\|f_z\|_{\mathrm{Lip}}}{n^{1/2}}\\
		&\leq
		\max\{C_0, C_1T^{3/2}\}
		\left\{
			\int_{-\varepsilon}^{\varepsilon}
			\frac{1}{\sigma(z+y)^2}
			\rd y
			+
			\frac{\|f_z\|_{\mathrm{Lip}}}{n^{1/2}}
		\right\},
	\end{align*}
	which implies \eqref{Lem_1_2}.
\end{proof}

Now we introduce the following key lemma which is proved by Skorokhod (see, e.g. Theorem in \cite{Sk65}, Chapter 3, section 3, page 32), and which shows convergence in probability for a sequence of stochastic integrals.

\begin{Lem}[Skorokhod, \cite{Sk65}]\label{Lem_Sk}
	Let $(W,\mathcal{F}^{W})$ and $(W^{n},\mathcal{F}^{W^{n}})$, $n \in \n$ be Brownian motions and $f_n=(f_n(t))_{t \in [0,T]}$ be a $\mathcal{F}^{W^{n}}$-adapted stochastic processes such that $\int_{0}^{t} f_n(s) \rd W_s^{n}$ is well-defined, for all $n \in \n$.
	Suppose that for any $t \in [0,T]$, $W^{n}_t$ and $f_n(t)$ converges to $W_t$ and a $\mathcal{F}^{W}$-adapted process $f(t)$ in probability, respectively, and the stochastic integral $\int_{0}^{t} f(s) \rd W_s$ is well-defined.
	Suppose further that the following conditions are satisfied for $\{f_n\}_{n \in \n}$:
	\begin{itemize}
		\item[(a)]
		For any $\varepsilon>0$, there exists $K>0$ such that for any $n \in \n$,
		\begin{align*}
			\p
			\left(
				\sup_{0 \leq t \leq T}
				|f_n(t)|
				>K
			\right)
			\leq \varepsilon.
		\end{align*}
		
		\item[(b)]
		For any $\varepsilon>0$,
		\begin{align*}
			\lim_{h \searrow 0}
			\lim_{n \to \infty}
			\sup_{|t_1-t_2|\leq h}
			\p
			\left(
				|f_n(t_2)-f_n(t_1)|
				>
				\varepsilon
			\right)
			=0.
		\end{align*}
	\end{itemize}
	Then it holds that for any $t \in [0,T]$
	\begin{align*}
		\lim_{n \to \infty}
		\int_{0}^{t}
			f_n(s)
		\rd W^{n}_s
		=
		\int_{0}^{t}
			f(s)
		\rd W_s,
	\end{align*}
	in probability.
\end{Lem}

Finally, we prove the following  key lemma in order to show main theorems and in particular to deal with non-sticky condition.

\begin{Lem}\label{Lem_2}
	Suppose that Assumption \ref{Ass_1} holds.
	Let $(X,W)$ be a solution of SDE \eqref{SDE_1} with non-sticky condition \eqref{boundary_0} and $\{X^{(k)}\}_{k \in \n}$ be a sub-sequence of the Euler--Maruyama scheme defined by \eqref{EM_0}.
	Then there exists a probability space
	$(\widehat{\Omega},\widehat{\mathcal{F}},\widehat{\p})$,
	a sub-sub-sequence $\{k_{\ell}\}_{\ell \in \n}$ and three-dimensional continuous processes $\widehat{Y}^{k_{\ell}}=(\widehat{X}^{k_{\ell}},\widehat{X}^{(k_{\ell})},\widehat{W}^{k_{\ell}})$ and $\widehat{Y}=(\widehat{X},\widehat{X}^{(*)},\widehat{W})$
	defined on the probability space $(\widehat{\Omega},\widehat{\mathcal{F}},\widehat{\p})$ such that the following properties are satisfied:
	\begin{itemize}
		\item[(i)]
		The law of stochastic processes $(X,X^{(k_{\ell})},W)$ and $(\widehat{X}^{k_{\ell}},\widehat{X}^{(k_{\ell})},\widehat{W}^{k_{\ell}})$ coincide for each $\ell \in \n$.
		In particular, $(\widehat{X}^{k_{\ell}},\widehat{X}^{(k_{\ell})},\widehat{W}^{k_{\ell}})$ can be chosen as follows: there exist measurable maps $\phi_{k_{\ell}}:\widehat{\Omega} \to \Omega$, $\ell \in \n$ such that
		\begin{align*}
		(
			\widehat{X}^{k_{\ell}},
			\widehat{X}^{(k_{\ell})},
			\widehat{W}^{k_{\ell}}
		)
		=
		(
			X \circ \phi_{k_{\ell}},
			X^{(k_{\ell})} \circ \phi_{k_{\ell}},
			W \circ \phi_{k_{\ell}}
		).
		\end{align*}
		
		\item[(ii)]
		$\widehat{\p}(
		\lim_{\ell \to \infty}
		\sup_{0 \leq t \leq T}|\widehat{Y}_t^{k_{\ell}}-\widehat{Y}_t|
		=0
		)
		=1$.
		
		\item[(iii)]
		$\widehat{W}$ is a Brownian motion and $\widehat{X}$, $\widehat{X}^{(*)}$ are continuous martingales on $(\widehat{\Omega},\widehat{\mathcal{F}},\widehat{\p})$.
		
		\item[(iv)]
		$\widehat{X}$ and $\widehat{X}^{(*)}$ satisfy non-sticky condition
		\begin{align}\label{boundary_1}
			\widehat{\e}
				\left[
					\int_{0}^{T}
						\1_{Z(\sigma)}(\widehat{X}_s)
					\rd s
				\right]
			=
			\widehat{\e}
				\left[
					\int_{0}^{T}
						\1_{Z(\sigma)}(\widehat{X}_s^{(*)})
					\rd s
				\right]
			=
			0.
		\end{align}
		
		\item[(v)]
		There exist an extension $(\widetilde{\Omega},\widetilde{\mathcal{F}},\widetilde{\p})$ of $(\widehat{\Omega},\widehat{\mathcal{F}},\widehat{\p})$ and Brownian motions $\widetilde{B}=(\widetilde{B}_t)_{t \in [0,T]}$, $\widetilde{B}^{(*)}=(\widetilde{B}_t^{(*)})_{t \in [0,T]}$ such that $(\widehat{X},\widetilde{B})$ and $(\widehat{X}^{(*)},\widetilde{B}^{(*)})$ are solutions of SDE \eqref{SDE_1}  with non-sticky condition \eqref{boundary_1}.
		
		\item[(vi)]
		If $\sigma$ is continuous then $(\widehat{X},\widehat{W})$ and $(\widehat{X}^{(*)},\widehat{W})$ are solutions of SDE \eqref{SDE_1} with non-sticky condition \eqref{boundary_1}.
	\end{itemize}
\end{Lem}
\begin{proof}
	Proof of (i) and (ii).
	We first note that since the diffusion coefficient $\sigma$ is of linear growth, the estimates in Lemma \ref{Lem_0} hold.
	Hence it follows from Theorem 4.3 and the proof of Theorem 4.2 in \cite{IkWa} that the family of three-dimensional stochastic process $\{(X,X^{(k)},W)\}_{k \in \n}$ is tight in $C[0,T]^{k}$, and thus is relatively compact in $C[0,T]^{k}$ by Prohorov's Theorem (see, e.g. Theorem 2.4.7 in \cite{KS}).
	Hence there exist a sub-sequence $\{k_{\ell}\}_{\ell \in \n}$ and $X^{(*)}$ such that $\lim_{\ell \to \infty}\e[f(X,X^{(k_{\ell})},W)]=\e[f(X,X^{(*)},W)]$, for any $f \in C_b(C[0,T]^{3};\real)$.
	Therefore, by using Skorohod's representation theorem (see, e.g. Theorem 1.2.7 in \cite{IkWa} or Theorem 1.10.4 in \cite{VaWe96}) and Addendum 1.10.5 in \cite{VaWe96}, there exists a probability space $(\widehat{\Omega},\widehat{\mathcal{F}},\widehat{\p})$, three-dimensional continuous processes $\widehat{Y}^{k_{\ell}}=(\widehat{X}^{k_{\ell}},\widehat{X}^{(k_{\ell})},\widehat{W}^{k_{\ell}})$ and $\widehat{Y}=(\widehat{X},\widehat{X}^{(*)},\widehat{W})$
	defined on the probability space $(\widehat{\Omega},\widehat{\mathcal{F}},\widehat{\p})$ and measurable maps $\phi_{k_{\ell}}:\widehat{\Omega} \to \Omega$, $\ell \in \n$ such that the properties (i) and (ii) are satisfied.
	
	Proof of (iii).
	We first prove $\widehat{W}$ is a Brownian motion on $(\widehat{\Omega},\widehat{\mathcal{F}},\widehat{\p})$.
	From the property (i), $\widehat{W}^{k_{\ell}}$ is a Brownian motion, so $\widehat{W}^{k_{\ell}}$ and $(|\widehat{W}_t^{k_{\ell}}|^2-t)_{t \in [0,T]}$ are martingales.
	Therefore it follows from Lemma A.1 in \cite{Yan} and the above property (ii) that $\widehat{W}$ and $(|\widehat{W}|^2_t-t)_{t \in [0,T]}$ are martingales, thus the quadratic variation of $\widehat{W}_t$ is $t$ for all $t \in [0,T]$.
	L\'evy's Theorem (e.g. Theorem 3.3.16 in \cite{KS}) implies that $\widehat{W}$ is a Brownian motion.
	
	Next, we prove $\widehat{X}$ and $\widehat{X}^{(*)}$ are continuous martingales.
	By using the above property (i), it holds that $(\widehat{X}^{k_{\ell}}, \widehat{W}^{k_{\ell}})$ satisfies the following equations
	\begin{align}\label{Lem_2_3_1}
		\widehat{X}_t^{k_{\ell}}
		&=
		x_0
		+\int_{0}^{t}
			\sigma(\widehat{X}_s^{k_{\ell}})
		\rd \widehat{W}_s^{k_{\ell}}
		\quad\text{and}\quad
		\widehat{\e}
		\left[
			\int_{0}^{T}
				\1_{Z(\sigma)}(\widehat{X}_s^{k_{\ell}})
			\rd s
		\right]
		=0
	\end{align}
	and by using Lemma \ref{Lem_EM_0}, $(\widehat{X}^{(k_{\ell})}, \widehat{W}^{k_{\ell}})$ satisfies the following equations
	\begin{align}\label{Lem_2_3_2}
		\widehat{X}^{(k_{\ell})}_t
		&=
		x_{k_{\ell}}
		+\int_{0}^{t}
			\sigma(\widehat{X}^{(k_{\ell})}_{\eta_{k_{\ell}}(s)})
		\rd \widehat{W}_s^{k_{\ell}}
		\quad\text{and}\quad
		\widehat{\e}
			\left[
				\int_{0}^{T}
					\1_{Z(\sigma)}(\widehat{X}_s^{(k_{\ell})})
				\rd s
			\right]
		=0.
	\end{align}
	Thus from Lemma \ref{Lem_0}, sequences of stochastic process $\{\widehat{X}^{k_{\ell}} \}_{\ell \in \n}$ and $\{\widehat{X}^{(k_{\ell})} \}_{\ell \in \n}$ are uniformly integrable martingales, which uniformly converge to $\widehat{X}$ and $\widehat{X}^{(*)}$, respectively.
	Hence from Lemma A.1 in \cite{Yan}, we conclude $\widehat{X}$ and $\widehat{X}^{(*)}$ are continuous martingales.
	
	Proof of (iv).
	For $\varepsilon>0$ and $z \in Z(\sigma)$, we define a continuous function $f_{\varepsilon,z}:\real \to [0,1]$ by
	\begin{align*}
		f_{\varepsilon,z}(x)
		:=
		\left\{\begin{array}{ll}
		\displaystyle
		-\frac{x-z}{\varepsilon}+1
		&\text{ if } 0\leq x-z<\varepsilon,  \\
		\displaystyle
		\frac{x-z}{\varepsilon}+1 &\text{ if } -\varepsilon < x-z <0,  \\
		\displaystyle
		0 &\text{ if } |x-z| \geq \varepsilon.
		\end{array}\right.
	\end{align*}
	Then it is easy to see that $\lim_{\varepsilon \searrow 0} f_{\varepsilon,z}(x)=\1_{\{z\}}(x)$ for each $x \in \real$, and $f_{\varepsilon,z}$ is Lipschitz continuous with $\|f_{\varepsilon,z}\|_{\mathrm{Lip}}=2/\varepsilon$.
	Recall that $\widehat{X}^{k_{\ell}}$ and $\widehat{X}^{(k_{\ell})}$ satisfy the equations \eqref{Lem_2_3_1} and \eqref{Lem_2_3_2}, respectively.
	Hence from the dominated convergence theorem, Lemma \ref{Lem_1} and Assumption \ref{Ass_1} (i), we have
	\begin{align*}
		&
		\widehat{\e}
		\left[
			\int_{0}^{T}
				\1_{Z(\sigma)}(\widehat{X}_s)
			\rd s
		\right]
		+
		\widehat{\e}
		\left[
			\int_{0}^{T}
				\1_{Z(\sigma)}(\widehat{X}^{(*)}_s)
			\rd s
		\right]\\
		&=
		\sum_{z \in Z(\sigma)}
		\left\{
			\widehat{\e}
			\left[
				\int_{0}^{T}
				\1_{\{z\}}(\widehat{X}_s)
			\rd s
			\right]
			+
			\widehat{\e}
			\left[
				\int_{0}^{T}
				\1_{\{z\}}(\widehat{X}^{(*)}_s)
			\rd s
			\right]
		\right\}\\
		&=
		\sum_{z \in Z(\sigma)}
		\lim_{\varepsilon \searrow 0}
		\left\{
			\widehat{\e}
			\left[
			\int_{0}^{T}
				f_{\varepsilon,z}(\widehat{X}_s)
			\rd s
			\right]
			+
			\widehat{\e}
			\left[
				\int_{0}^{T}
					f_{\varepsilon,z}(\widehat{X}^{(*)}_s)
				\rd s
			\right]
			\right\}
			\\
		&=
		\sum_{z \in Z(\sigma)}
		\lim_{\varepsilon \searrow 0}
		\lim_{\ell \to \infty}
		\left\{
			\widehat{\e}
				\left[
					\int_{0}^{T}
						f_{\varepsilon,z}(\widehat{X}_{s}^{k_{\ell}})
					\rd s
				\right]
			+
			\widehat{\e}
				\left[
				\int_{0}^{T}
				f_{\varepsilon,z}(\widehat{X}^{(k_{\ell})}_s)
				\rd s
			\right]
		\right\}\\
		&\leq
		2C
		\sum_{z \in Z(\sigma)}
		\lim_{\varepsilon \searrow 0}
		\lim_{\ell \to \infty}
		\left\{
			\int_{-\varepsilon}^{\varepsilon}
				\frac{1}{\sigma(z+y)^2}
			\rd y
			+
			\frac{1}{\varepsilon k_{\ell}^{1/2}}
		\right\}\\
		&
		=
		2C
		\sum_{z \in Z(\sigma)}
		\lim_{\varepsilon \searrow 0}
			\int_{-\varepsilon}^{\varepsilon}
			\frac{1}{\sigma(z+y)^2}
			\rd y
		=0,
	\end{align*}
	which concludes (iv).
	
	Proof of (v).
	The proof is almost the same as Lemma 2.3 and Theorem 2.1 in \cite{Yan}.
	We first prove that for each $t \in [0,T]$,
	\begin{align}\label{Lem_2_1}
		\lim_{\ell \to \infty}
		\langle \widehat{X}^{k_{\ell}} \rangle_t
		=
		\langle \widehat{X} \rangle_t
		\quad\text{and}\quad
		\lim_{\ell \to \infty}
		\langle \widehat{X}^{(k_{\ell})} \rangle_t
		=
		\langle \widehat{X}^{(*)} \rangle_t
	\end{align}
	in $L^{1}(\widehat{\Omega},\widehat{\mathcal{F}},\widehat{\p})$ and
	\begin{align}
		\lim_{\ell \to \infty}
				\int_{0}^{t}
					\sigma(\widehat{X}_s^{k_{\ell}})^2
					\1(\widehat{X}_s \notin D(\sigma))
					\rd s
		&=
				\int_{0}^{t}
					\sigma(\widehat{X}_s)^2
					\1(\widehat{X}_s \notin D(\sigma))
				\rd s,
		\label{Lem_2_2}
		\\
		\lim_{\ell \to \infty}
				\int_{0}^{t}
					\sigma(\widehat{X}_{\eta_{k_{\ell}}(s)}^{(k_{\ell})})^2
					\1(\widehat{X}_s^{(*)} \notin D(\sigma))
				\rd s
		&=
				\int_{0}^{t}
					\sigma(\widehat{X}_{s}^{(*)})^2
					\1(\widehat{X}_s^{(*)} \notin D(\sigma))
				\rd s,
		\label{Lem_2_3}
	\end{align}
	in $L^{1}(\widehat{\Omega},\widehat{\mathcal{F}},\widehat{\p})$.
	From the property (ii), continuous martingales $\widehat{X}^{k_{\ell}}$ and $\widehat{X}^{(k_{\ell})}$ converge to $\widehat{X}$ and $\widehat{X}^{(*)}$ almost surely in $C[0,T]$, respectively.
	Hence it follows from Theorem 2.2 in \cite{KuPr91} that $\int_{0}^{\cdot} \widehat{X}^{k_{\ell}}_s \rd \widehat{X}^{k_{\ell}}_s$ and $\int_{0}^{\cdot} \widehat{X}^{(k_{\ell})}_s \rd \widehat{X}^{(k_{\ell})}_s$ converge to $\int_{0}^{\cdot} \widehat{X}_s \rd \widehat{X}_s$ and $\int_{0}^{\cdot} \widehat{X}^{(*)}_s \rd \widehat{X}^{(*)}_s$ in probability as $\ell \to \infty$, respectively.
	Since for any squared integrable continuous martingale $M$, $\langle M \rangle_t=M_t^2-M_0^2-2\int_{0}^{t} M_s \rd M_s$, we have
	\begin{align*}
		\lim_{\ell \to \infty}
		\langle \widehat{X}^{k_{\ell}} \rangle_t
		=
		\langle \widehat{X} \rangle_t
		\quad\text{and}\quad
		\lim_{\ell \to \infty}
		\langle \widehat{X}^{(k_{\ell})} \rangle_t
		=
		\langle \widehat{X}^{(*)} \rangle_t,
	\end{align*}
	in probability.
	On the other hand, since $\sigma$ is of linear growth, from Lemma \ref{Lem_0}, the classes
	\begin{align*}
		\left\{
			\langle \widehat{X}^{k_{\ell}} \rangle_t,~
			\langle \widehat{X}^{(k_{\ell})} \rangle_t
			~;~
			t \in [0,T],~
			\ell \in \n
		\right\}
		~\text{and}~
		\left\{
			\sigma(\widehat{X}^{k_{\ell}}_t)^2,~
			\sigma(\widehat{X}^{(k_{\ell})}_{\eta_{k_{\ell}}(t)})^2
		~;~
		t \in [0,T],~
		\ell \in \n
		\right\}
	\end{align*}
	are uniformly integrable, thus we conclude \eqref{Lem_2_1}, \eqref{Lem_2_2} and \eqref{Lem_2_3}.
	
	Recall that $\sigma_1(y)^2:=\liminf_{x \to y} \sigma(x)^2$ for $y \in \real$.
	Using Fatou's lemma and \eqref{Lem_2_1}, we have for any $0 \leq r < t \leq T$,
	\begin{align*}
		\widehat{\e}
		\left[
			\int_{r}^{t}
				\sigma_1(\widehat{X}_s)^2
			\rd s
		\right]
		=
		\widehat{\e}
			\left[
			\int_{r}^{t}
				\liminf_{\ell \to \infty}
				\sigma(\widehat{X}_s^{k_{\ell}})^2
			\rd s
		\right]
		\leq
		\liminf_{\ell \to \infty}
		\widehat{\e}
			\left[
			\int_{r}^{t}
				\sigma(\widehat{X}_s^{k_{\ell}})^2
			\rd s
		\right]
		=
		\widehat{\e}
		\left[
			\langle \widehat{X} \rangle_t
			-
			\langle \widehat{X} \rangle_r
		\right]
	\end{align*}
	and by the same way,
	\begin{align*}
		\widehat{\e}
		\left[
		\int_{r}^{t}
			\sigma_1(\widehat{X}_s^{(*)})^2
		\rd s
		\right]
		\leq
		\widehat{\e}
			\left[
				\langle \widehat{X}^{(*)} \rangle_t
				-
				\langle \widehat{X}^{(*)} \rangle_r
			\right].
	\end{align*}
	Therefore, using the occupation time formula and Lemma \ref{Lem_0}, we have
	\begin{align*}
		\widehat{\e}
			\left[
				\int_{0}^{T}
					\1_{D(\sigma)}(\widehat{X}_s)
					\sigma_1(\widehat{X}_s)^2
				\rd s
			\right]
		&\leq
		\widehat{\e}
		\left[
			\int_{0}^{T}
				\1_{D(\sigma)}(\widehat{X}_s)
			\rd \langle \widehat{X} \rangle_s
		\right]
		=
		\widehat{\e}
		\left[
			\int_{D(\sigma)}
				L_T^y(\widehat{X})
			\rd y
		\right]\\
		&
		\leq
		\mathrm{Leb}(D(\sigma))
		\sup_{y \in \real}
		\widehat{\e}[L_T^y(\widehat{X})]
		=0
	\end{align*}
	and by the same way
	\begin{align*}
		\widehat{\e}
			\left[
				\int_{0}^{T}
					\1_{D(\sigma)}(\widehat{X}_s^{(*)})
					\sigma_1(\widehat{X}_s^{(*)})^2
				\rd s
			\right]
		=0.
	\end{align*}
	Recall that from the assumption, $\sigma_1(y)^2=\liminf_{x \to y} \sigma(x)^2>0$ for $y \in D(\sigma)$, so we obtain
	\begin{align}\label{Lem_2_3.5}
		\int_{0}^{T}
			\1_{D(\sigma)}(\widehat{X}_s)
		\rd s
		=
		\int_{0}^{T}
			\1_{D(\sigma)}(\widehat{X}_s^{(*)})
		\rd s
		=0,
	\end{align}
	$\widehat{\p}$-almost surely.
	Therefore, it hold  from \eqref{Lem_2_2}, \eqref{Lem_2_3.5} and \eqref{Lem_2_3} that for any $t \in [0,T]$,
	\begin{align}\label{Lem_2_4}
		\lim_{\ell \to \infty}
		\widehat{\e}
		\left[
			\left|
				\langle \widehat{X}^{k_{\ell}} \rangle_t
				-
				\int_{0}^{t}
					\sigma(\widehat{X}_s)^2
				\rd s
			\right|
		\right]
		&
		=
		\lim_{\ell \to \infty}
		\widehat{\e}
		\left[
			\left|
				\int_{0}^{t}
					\left\{	
						\sigma(\widehat{X}_s^{k_{\ell}})^2
						-
						\sigma(\widehat{X}_s)^2
					\right\}
					\1(\widehat{X}_s \notin D(\sigma))
				\rd s
			\right|
		\right]
		=0
	\end{align}
	and
	\begin{align}\label{Lem_2_5}
		\lim_{\ell \to \infty}
		\widehat{\e}
		\left[
			\left|
				\langle \widehat{X}^{(k_{\ell})} \rangle_t
				-
				\int_{0}^{t}
					\sigma(\widehat{X}^{(*)}_s)^2
				\rd s
			\right|
		\right]
		=
		0.
	\end{align}
	Therefore, from \eqref{Lem_2_1}, \eqref{Lem_2_4} and \eqref{Lem_2_5}, we obtain
	$\langle \widehat{X} \rangle_t=\int_{0}^{t}\sigma(\widehat{X}_s)^2\rd s$ and $\langle \widehat{X}^{(*)} \rangle_t=\int_{0}^{t}\sigma(\widehat{X}_s^{(*)})^2\rd s$, $\widehat{\p}$-almost surely.
	Since $\widehat{X}$ and $\widehat{X}^{(*)}$ are square integrable continuous martingales, by using martingale representation theorem (see, e.g. chapter 2, Theorem 7.1' in \cite{IkWa}), there exist an extension $(\widetilde{\Omega},\widetilde{\mathcal{F}},\widetilde{\p})$ of $(\widehat{\Omega},\widehat{\mathcal{F}},\widehat{\p})$ and Brownian motions $\widetilde{B}=(\widetilde{B}_t)_{t \in [0,T]}$, $\widetilde{B}^{(*)}=(\widetilde{B}_t^{(*)})_{t \in [0,T]}$ such that
	\begin{align*}
		\widehat{X}_t
		=
		x_0
		+
		\int_{0}^{t}
			\sigma(\widehat{X}_s)
		\rd \widetilde{B}_s
		\quad\text{and}\quad
		\widehat{X}_t^{(*)}
		=
		x_0
		+
		\int_{0}^{t}
			\sigma(\widehat{X}_s^{(*)})
		\rd \widetilde{B}_s^{(*)},
	\end{align*}
	$\widetilde{\p}$-almost surely, thus from the property (iv), we conclude the statement of (v).
	
	Proof of (vi).
	We first prove that $\{\widehat{X}^{k_{\ell}}\}_{\ell \in \n}$ and $\{\widehat{X}^{(k_{\ell})}\}_{\ell \in \n}$ satisfy the following two properties:
	\begin{itemize}
		\item[(a)]
		For any $\varepsilon>0$ and a measurable, polynomial growth function $f:\real \to \real$, there exists $K \equiv K(\varepsilon,f)>0$ such that
		\begin{align*}
			\sup_{\ell \in \n}
			\max
			\left\{
				\widehat{\p}
				\left(
					\sup_{0 \leq s \leq T}
						|f(\widehat{X}_s^{k_{\ell}})|
					\geq K
				\right)
				,
				\widehat{\p}
				\left(
					\sup_{0 \leq s \leq T}
						|f(\widehat{X}_{s}^{(k_{\ell})})|
					\geq K
				\right)
			\right\}
			< \varepsilon.
		\end{align*}
		\item[(b)]
		For any $\widetilde{\varepsilon}>0$ and a continuous function $g:\real \to \real$,
		\begin{align*}
			\lim_{h \to 0}
			\lim_{\ell \to \infty}
			\sup_{|t_1-t_2|\leq h}
			\max
			\left\{
				\widehat{\p}
				\left(
					\left|
						g(\widehat{X}_{t_1}^{k_{\ell}})
						-
						g(\widehat{X}_{t_2}^{k_{\ell}})
					\right|
					>\widetilde{\varepsilon}
				\right)
				,
				\widehat{\p}
				\left(
					\left|
						g(\widehat{X}_{t_1}^{(k_{\ell})})
						-
						g(\widehat{X}_{t_2}^{(k_{\ell})})
					\right|
					>\widetilde{\varepsilon}
				\right)
			\right\}
			=0
		\end{align*}
	\end{itemize}
	Indeed, the property (a) follows from Markov's inequality and Lemma \ref{Lem_0}.
	In order to prove the property (b), we use the property (a) with $f(x)=x$.
	Then since $g$ is uniformly continuous on the interval $[-K,K]$, there exists $\delta \equiv \delta(\widetilde{\varepsilon},K)>0$ such that for any $x,y \in [-K,K]$, if $|x-y|<\delta$ then $|g(x)-g(y)|<\widetilde{\varepsilon}$.
	Therefore, it follows from Markov's inequality and Lemma \ref{Lem_0} that
	\begin{align*}
		&
		\sup_{\ell \in \n}
		\widehat{\p}
			\left(
				\left|
					g(\widehat{X}_{t_1}^{k_{\ell}})
					-
					g(\widehat{X}_{t_2}^{k_{\ell}})
				\right|
				\geq \widetilde{\varepsilon} 
			\right)\\
		&
		\leq
			\sup_{\ell \in \n}
			\widehat{\p}
				\left(
					\left|
						g(\widehat{X}_{t_1}^{k_{\ell}})
						-
						g(\widehat{X}_{t_2}^{k_{\ell}})
					\right|
					\geq \widetilde{\varepsilon},~
					\sup_{s\in[0,T]}|\widehat{X}_{s}^{k_{\ell}}|
					\leq
					K
				\right)
		+
			\sup_{\ell \in \n}
			\widehat{\p}
			\left(
				\sup_{s\in[0,T]}|\widehat{X}_{s}^{k_{\ell}}|
				\geq
				K
			\right)
		\\
		&
		\leq
			\sup_{\ell \in \n}
			\widehat{\p}
				\left(
					\left|
						\widehat{X}_{t_1}^{k_{\ell}}
						-
						\widehat{X}_{t_2}^{k_{\ell}}
					\right|
					\geq \delta
				\right)
		+\varepsilon
		\leq
		\frac{C_1|t_1-t_2|^{1/2}}{\delta}
		+
		\varepsilon
	\end{align*}
	and by the same way,
	\begin{align*}
		\sup_{\ell \in \n}
		\widehat{\p}
			\left(
				\left|
					g(\widehat{X}_{t_1}^{(k_{\ell})})
					-
					g(\widehat{X}_{t_2}^{(k_{\ell})})
				\right|
				\geq \widetilde{\varepsilon}
			\right)
		\leq
		\frac{C_1|t_1-t_2|^{1/2}}{\delta}
		+
		\varepsilon.
	\end{align*}
	By taking $h \to 0$, since $\varepsilon$ is arbitrary, the property (b) follows.
	
	Recall that $\sigma$ is continuous, $\lim_{n \to \infty} x_n=x_0$ and $\widehat{W}$ is a Brownian motion.
	It follows from Lemma \ref{Lem_Sk} and the above properties (a), (b) with $f=g=\sigma$ that, by letting $\ell \to \infty$, the limits $\widehat{X}$ and $\widehat{X}^{(*)}$ are satisfies the equation
	\begin{align*}
		\widehat{X}_t
		&=
		x_0
		+\int_{0}^{t}
			\sigma(\widehat{X}_s)
		\rd \widehat{W}_s
		\quad\text{and}\quad
		\widehat{X}^{(*)}_t
		=
		x_0
		+\int_{0}^{t}
			\sigma(\widehat{X}^{(*)}_s)
		\rd \widehat{W}_s
	\end{align*}
	and thus from the property (iv), we conclude the statement of (vi).
\end{proof}

\subsection{Proof of main theorems}\label{sub_sec_2_4}

Before proving Theorem \ref{main_0}, we recall the following elementally fact on calculus.
Let $\{a_n\}_{n \in \n}$ be a sequence on $\real$ and $a \in \real$.
If for any sub-sequence $\{a_{n_{k}}\}_{k \in \n}$ of $\{a_n\}_{n \in \n}$, there exists a sub-sub-sequence $\{a_{n_{k_{\ell}}}\}_{\ell \in \n}$ such that $\lim_{\ell \to \infty} a_{n_{k_{\ell}}}=a$, then the sequence $\{a_n\}_{n \in \n}$ converges to $a$.
By using the this fact and Lemma \ref{Lem_2}, we prove Theorem \ref{main_0}.

\begin{proof}[Proof of Theorem \ref{main_0}]
	It is enough to prove that for any sub-sequence $\{X^{(k)}\}_{k \in \n}$ of the Euler--Maruyama scheme $\{X^{(n)}\}_{n \in \n}$ defined by \eqref{EM_0}, there is a sub-sub-sequence $\{X^{(k_{\ell})}\}_{\ell \in \n}$ such that  for any $f \in C_b(C[0,T];\real)$,
	\begin{align*}
		\lim_{\ell \to \infty}
		\e[f(X^{(k_{\ell})})]
		=
		\e[f(X)].
	\end{align*}
	
	Let $\{X^{(k)}\}_{k \in \n}$ be a sub-sequence of the Euler--Maruyama scheme $\{X^{(n)}\}_{n \in \n}$.
	From Lemma \ref{Lem_2}, there exists a probability space
	$(\widehat{\Omega},\widehat{\mathcal{F}},\widehat{\p})$,
	a sub-sequence $\{k_{\ell}\}_{\ell \in \n}$ and $3$-dimensional continuous processes $\widehat{Y}^{k_{\ell}}=(\widehat{X}^{k_{\ell}},\widehat{X}^{(k_{\ell})},\widehat{W}^{k_{\ell}})$ and $\widehat{Y}=(\widehat{X},\widehat{X}^{(*)},\widehat{W})$
	defined on the probability space $(\widehat{\Omega},\widehat{\mathcal{F}},\widehat{\p})$ such that the properties (i)--(v) are satisfied.
	
	For the proof of this theorem, we only use $\widehat{X}^{(k_{\ell})}$ and $\widehat{X}^{(*)}$, do not use $(\widehat{X}^{k_{\ell}},\widehat{W}^{k_{\ell}})$ and $(\widehat{X},\widehat{W})$.
	From the property (i), (ii) in Lemma \ref{Lem_2} and using the dominated convergence theorem, we have
	\begin{align}\label{weak_main_0}
		\lim_{\ell \to \infty}
		\e[f(X^{(k_{\ell})})]
		=
		\lim_{\ell \to \infty}
		\widehat{\e}[f(\widehat{X}^{(k_{\ell})})]
		=\widehat{\e}[f(\widehat{X}^{(*)})],
	\end{align}
	for any $f \in C_b(C[0,T];\real)$.
	
	On the other hand, the property (iv) and (v) imply that there exist an extension $(\widetilde{\Omega},\widetilde{\mathcal{F}},\widetilde{\p})$ of $(\widehat{\Omega},\widehat{\mathcal{F}},\widehat{\p})$ and Brownian motion $\widetilde{B}^{(*)}=(\widetilde{B}_t^{(*)})_{t \in [0,T]}$ such that $(\widehat{X}^{(*)},\widetilde{B}^{(*)})$ is a solution of SDE \eqref{SDE_1} with non-sticky condition \eqref{boundary_1}.
	Hence from the uniqueness in law for SDE \eqref{SDE_1} with non-sticky condition \eqref{boundary_1} (see, Remark \ref{Rem_0} (i)) and \eqref{weak_main_0}, we have
	\begin{align*}
		\lim_{\ell \to \infty}
		\e[f(X^{(k_{\ell})})]
		=\widehat{\e}[f(\widehat{X}^{(*)})]
		=\widetilde{\e}[f(\widehat{X}^{(*)})]
		=\e[f(X)]
	\end{align*}
	for any $f \in C_b(C[0,T];\real)$.
	This concludes the statement.
\end{proof}

\begin{proof}[Proof of Theorem \ref{main_1}]
	The proof for the statement (ii) is similar to (i), thus we only prove the statement (i).
	
	The proof is based on \cite{KaNa}, that is, we prove the statement by contradiction.
	We suppose that the statement (i) is not true, that is, there exist $\varepsilon_0>0$ and a sub-sequence $\{n_k\}_{k \in \n}$ such that
	\begin{align}\label{pr_main_1}
		\e\left[
			\sup_{0 \leq t \leq T}
			\left|
				X_t
				-
				X_t^{(n_k)}
			\right|^p
		\right]
		\geq \varepsilon_0,
		~\text{for any}~k \in \n.
	\end{align}
	We now denote $X^{(k)}$ by $X^{(n_k)}$ to simplify.
	Then from Lemma \ref{Lem_2}, there exist a probability space
	$(\widehat{\Omega},\widehat{\mathcal{F}},\widehat{\p})$,
	a sub-sequence $\{k_{\ell}\}_{\ell \in \n}$ and $3$-dimensional continuous processes $\widehat{Y}^{k_{\ell}}=(\widehat{X}^{k_{\ell}},\widehat{X}^{(k_{\ell})},\widehat{W}^{k_{\ell}})$ and $\widehat{Y}=(\widehat{X},\widehat{X}^{(*)},\widehat{W})$
	defined on the probability space $(\widehat{\Omega},\widehat{\mathcal{F}},\widehat{\p})$ such that the properties (i)--(vi) are satisfied.
	
	Note that from Lemma \ref{Lem_0}, the family of random variables $\{\sup_{0 \leq t \leq T}|\widehat{X}_{t}^{k_{\ell}}-\widehat{X}_t^{(k_{\ell})}|^p\}_{\ell \in \n}$ is uniformly integrable.
	Therefore, from the assumption \eqref{pr_main_1} and the property (i), (ii) in Lemma \ref{Lem_2}, we have
	\begin{align}\label{pr_main_0}
		\varepsilon_0
		&
		\leq
		\liminf_{\ell \to \infty}
			\e
				\left[
					\sup_{0 \leq t \leq T}
					\left|
						X_t
						-
						X_t^{(k_{\ell})}
					\right|^p
				\right] \notag\\
		&=
		\liminf_{\ell \to \infty}
			\widehat{\e}
				\left[
					\sup_{0 \leq t \leq T}
					\left|
						\widehat{X}_{t}^{k_{\ell}}
						-
						\widehat{X}_t^{(k_{\ell})}
					\right|^p
				\right] \notag\\
		&=
		\widehat{\e}
			\left[
				\sup_{0 \leq t \leq T}
				\left|
					\widehat{X}_{t}
					-
					\widehat{X}_t^{(*)}
				\right|^p
			\right].
	\end{align}
	On the other hand, the property (iv) and (vi) imply that $\widehat{X}$ and $\widehat{X}^{(*)}$ are solutions of SDE \eqref{SDE_1} driven by the same Brownian motion $\widehat{W}$, with non-sticky condition \eqref{boundary_1} on the probability space 
	$(\widehat{\Omega},\widehat{\mathcal{F}},\widehat{\p})$.
	Hence from the assumption on the pathwise uniqueness, \eqref{pr_main_1} and \eqref{pr_main_0}, we conclude $0<\varepsilon_0 \leq 0$.
	This is the contradiction.
\end{proof}

\section*{Acknowledgements}
The authors would like to thank Professor Masatoshi Fukushima for his valuable comments.
The authors would also like to thank an anonymous referee for his/her careful readings and advices.
The first author was supported by JSPS KAKENHI Grant Number 17H06833.
The second author was supported by Sumitomo Mitsui Banking Corporation.

\end{document}